\documentclass[11pt]{amsart}

\usepackage{etex}
\usepackage{amsmath, amssymb}
\usepackage[frame,cmtip,arrow,matrix,line,graph,curve]{xy}
\usepackage{graphpap, color, paralist, pstricks}
\usepackage[mathscr]{eucal}
\usepackage[pdftex]{graphicx}
\usepackage[pdftex,colorlinks,backref=page,citecolor=blue]{hyperref}
\usepackage{tikz}

\newtheorem{theorem}{Theorem}[section]
\newtheorem{proposition}[theorem]{Proposition}
\newtheorem{corollary}[theorem]{Corollary}
\newtheorem{lemma}[theorem]{Lemma}

\theoremstyle{definition}
\newtheorem{definition}[theorem]{Definition}

\newcommand{\PP}{\mathbb{P}}
\newcommand{\QQ}{\mathbb{Q}}
\newcommand{\CC}{\mathbb{C}}

\newcommand{\cO}{\mathcal{O} }

\newcommand{\cE}{\mathcal{E} }
\newcommand{\cF}{\mathcal{F} }

\newcommand{\cN}{\mathcal{N} }
\newcommand{\cL}{\mathcal{L} }
\newcommand{\cQ}{\mathcal{Q} }
\newcommand{\cS}{\mathcal{S} }
\newcommand{\cT}{\mathcal{T} }
\newcommand{\cV}{\mathcal{V} }
\newcommand{\cU}{\mathcal{U} }
\newcommand{\cW}{\mathcal{W} }
\newcommand{\cZ}{\mathcal{Z} }

\newcommand{\rH}{\mathrm{H} }

\newcommand{\cHom}{\mathcal{H}om}

\def\Hom{\mathrm{Hom} }
\def\Ext{\mathrm{Ext} }
\def\cExt{\mathcal{E}xt }

\def\GL{\mathrm{GL}}
\def\git{/\!/ }

\def\lr{\rightarrow}

\def\bG{\mathbf{G}}
\def\bH{\mathbf{H}}
\def\bM{\mathbf{M}}

\def\bP{\mathbf{P}}

\newcommand{\ses}[3]{0\lr{#1}\lr{#2}\lr{#3}\lr 0}

\hypersetup{%
pdftitle={Birational Geometry of the Moduli Space of Pure Sheaves on Quadric Surface},%
pdfauthor={Kiryong Chung},%
pdfkeywords={moduli of sheaves, elementary modification, moduli space},%
citecolor=blue,%
linkcolor=blue,%
}

\begin{document}

\title{Birational geometry of the moduli space of pure sheaves on quadric surface}

\author{Kiryong Chung}
\address{Department of Mathematics Education, Kyungpook National University, 80 Daehakro, Bukgu, Daegu 41566, Korea}
\email{krchung@knu.ac.kr}

\author{Han-Bom Moon}
\address{Department of Mathematics, Fordham University, Bronx, NY 10458}
\email{hmoon8@fordham.edu}


\begin{abstract}
We study birational geometry of the moduli space of stable sheaves on a quadric surface with Hilbert polynomial $5m + 1$ and $c_{1} = (2, 3)$. We describe a birational map between the moduli space and a projective bundle over a Grassmannian as a composition of smooth blow-ups/downs. 
\end{abstract}

\maketitle


\section{Introduction}

The geometry of the moduli space of sheaves on a del Pezzo surface has been studied in various viewpoints, for instance curve counting, the strange duality conjecture, and birational geometry via Bridgeland stability. For a detailed description of the motivation, see \cite{CM15} and references therein. In this paper we continue the study of birational geometry of the moduli space of torsion sheaves on a del Pezzo surface, which was initiated in \cite{CM15}. More precisely, here we construct a flip between the moduli space of sheaves and a projective bundle, and show that their common blown-up space is the moduli space of stable pairs (\cite{LP93b}), in the case of a quadric surface.  

Let $Q \cong \PP^1\times \PP^1$ be a smooth quadric surface in $\PP^3$ with a very ample polarization $L := \cO_Q(1,1)$. For the convenience of the reader, we start with a list of relevant moduli spaces.

\begin{definition}
\begin{enumerate}
\item Let $\bM := \bM_{L}(Q, (2,3), 5m+1)$ be the moduli space of stable sheaves $F$ on $Q$ with $c_1(F)=c_1(\cO_Q(2,3))$ and $\chi(F(m))=5m+1$.
\item Let $\bM^{\alpha} := \bM^{\alpha}_{L}(Q, (2,3), 5m+1)$ be the moduli space of $\alpha$-stable pairs $(s,F)$ with $c_1(F)=c_1(\cO_Q(2,3))$ and $\chi(F(m))=5m+1$ (\cite{LP93b} and \cite[Theorem 2.6]{He98}).
\item Let $\bG = \mathrm{Gr}(2, 4)$ and let $\bG_{1}$ be the blow-up of $\bG$ along $\PP^{1}$ (Section \ref{sub:defq}).
\item Let $\bP:=\PP(\cU)$ and $\bP^{-} := \PP(\cU^{-})$, where $\cU$ (resp. $\cU^{-}$) is a rank $10$ vector bundle over $\bG$ (resp. $\bG_1$) defined in \eqref{eqn:exactseqforQ} in Section \ref{sub:defq} (resp. Section \ref{ssec:modification}).
\end{enumerate}
\end{definition}

The aim of this paper is to explain and justify the following commutative diagram between moduli spaces. 
\[
	\xymatrix{\bM^{+} \ar[r] \ar[d]_{r} & 
	\bP^{-} = \PP(\cU^{-})\ar@{<-->}[r] 
	\ar[rd] &
	\PP(u^{*}\cU) = \bG_{1} \times_{\bG}\bP \ar[d] \ar[r] 
	& \bP = \PP(\cU) \ar[d]\\
	\bM \ar@{<-->}[ru] && \bG_{1} \ar[r]^{u} 
	& \bG &}
\]
We have to explain two flips (dashed arrows) on the diagram. 

One of key ingredients is the \emph{elementary modification} of vector bundles (\cite{Mar73}), sheaves (\cite[Section 2.B]{HuLe10}), and pairs (\cite[Section 2.2]{CC12}). It has been widely used in the study of sheaves on a smooth projective variety. Let $\cF$ be a vector bundle on a smooth projective variety $X$ and $\cQ$ be a vector bundle on a smooth divisor $Z\subset X$ with a surjective map $\cF|_{Z}\twoheadrightarrow \cQ$. The elementary modification of $\cF$ along $Z$ is the kernel of the composition
\[
\mathrm{elm}_{Z}(\cF):=\mathrm{ker}(\cF\twoheadrightarrow \cF|_{Z}\twoheadrightarrow \cQ).
\]
A similar definition is valid for sheaves and pairs, too.

On $\bG_1$, let $\cU^{-} := \mathrm{elm}_{Y_{10}}(u^{*}\cU)$ be the elementary transformation of $u^*\cU$ along a smooth divisor $Y_{10}$ (Section \ref{sub:defq}).

\begin{proposition}\label{prop:downstairmodificaiton}
Let $\bP^{-} = \PP(\cU^{-})$. The flip $\bP^{-} \dashrightarrow \PP(u^{*}\cU) = \bG_{1} \times_{\bG}\PP(\cU)$ is a composition of a blow-up and a blow-down. The blow-up center in $\bP^{-}$ (resp. $\PP(u^{*}\cU)$) is a $\PP^{1}$ (resp. $\PP^{7}$)-bundle over $Y_{10}$.
\end{proposition}

\begin{theorem}\label{thm:mainpropintro}
There is a flip between $\bM$ and $\bP^{-}$ which is a blow-up followed by a blow-down, and the master space is $\bM^{+}$, the moduli space of $+$-stable pairs. 
\end{theorem}

As applications, we compute the Poincar\'e polynomial of $\bM$ and show the rationality of $\bM$ (Corollary \ref{coro}) which were obtained by Maican by different methods (\cite{Mai16}). Since each step of the birational transform is described in terms of blow-ups/downs along explicit subvarieties, in principle the cohomology ring and the Chow ring of $\bM$ can be obtained from that of $\bG$. Also one may aim for the completion of Mori's program for $\bM$. We will carry on these projects in forthcoming papers. 



\section{Relevant moduli spaces}\label{sec:relevantmoduli}

In this section we give definitions and basic properties of some relevant moduli spaces. 

\subsection{Grassmannian as a moduli space of Kronecker quiver representations}\label{sub:defq} 

The moduli space of representations of a Kronecker quiver parametrizes the isomorphism classes of stable sheaf homomorphisms 
\begin{equation}\label{res1}
\cO_Q(0,1)\longrightarrow \cO_Q(1,2)^{\oplus 2}
\end{equation}
up to the natural action of the automorphism group $\CC^{*}\times \GL_{2}/\CC^{*} \cong \GL_{2}$. For two vector spaces $E$ and $F$ of dimension $1$ and $2$ respectively and $V^{*} := \rH^{0}(Q, L)$, the moduli space is constructed as $\bG := \Hom(F, V^{*}\otimes E)\git \GL_{2} \cong V^{*}\otimes E \otimes F^{*}\git \GL_{2}$ with an appropriate linearization (\cite{Kin94}). Note that the $\GL_{2}$ acts as a row operation on the space of $2 \times 4$ matrices, $\bG \cong\mathrm{Gr}(2, 4)$. 

Let $\bH(n):=\mathrm{Hilb}^n(Q)$, the Hilbert scheme of $n$ points on $Q$. $\bH(2)$ is birational to $\bG$ because a general $Z \in \bH(2)$, $I_{Z}(2, 3)$ has a resolution of the form \eqref{res1}. For any $Z \in \bH(2)$, let $\ell_{Z}$ be the unique line in $\PP^{3} \supset Q$ containing $Z$. Then either $\ell_{Z} \cap Q = Z$ or $\ell_{Z} \subset Q$. In the second case, the class of $\ell_{Z}$ is of the type $(1, 0)$ or $(0, 1)$. Let $Y_{10}$ (resp. $Y_{01}$) be the locus of subschemes such that $\ell_{Z}$ is a line of the type $(1, 0)$ (resp. $(0, 1)$). Then $Y_{10}$ and $Y_{01}$ are two disjoint subvarieties which are isomorphic to a $\PP^{2}$-bundle over $\PP^{1}$. 

\begin{proposition}[\protect{\cite[Example 6.1]{BC13}}]\label{prop:hilbcon}
There exists a morphism $t:\bH(2)\longrightarrow \bG_1 \stackrel{u}\longrightarrow \bG$. The first (resp. the second) map contracts the divisor $Y_{01}$ (resp. $Y_{10}$) to $\PP^{1}$. If $\ell_{Z} \cap Q = Z$, then $t(Z) = I_{Z}(2, 3)$. If $Z \in Y_{10}$, then $t(Z) = E_{10} \in \PP(\Ext^{1}(\cO_{Q}(1, 3), \cO_{\ell_{Z}}(1))) = \{\mathrm{pt}\}$. If $Z \in Y_{01}$, then $t(Z) = E_{01} \in \PP(\Ext^{1}(\cO_{Q}(2, 2), \cO_{\ell_{Z}})) = \{\mathrm{pt}\}$. 
\end{proposition}

There is a \emph{universal morphism} $\phi : p_{1}^{*}\cF \otimes p_{2}^{*}\cO_{Q}(0,1)\to p_{1}^{*}\cE \otimes p_{2}^{*}\cO_{Q}(1,2)$ where $p_{1} : \bG \times Q \to \bG$ and $p_{2} : \bG \times Q \to Q$ are two projections (\cite{Kin94}). Let $\cU$ be the cokernel of $p_{1 *}\phi$. On the stable locus, $p_{1 *}\phi$ is injective. Thus we have an exact sequence
\begin{equation}\label{eqn:exactseqforQ}
	0 \to \cF \otimes \rH^{0}(\cO_{Q}(0,1)) \to\cE \otimes\rH^{0}(\cO_{Q}(1,2)) \to \cU \to 0
\end{equation}
and $\cU$ is a rank $10$ vector bundle. Let $\bP := \PP(\cU)$. 

\subsection{Moduli space $\bM$ of stable sheaves}\label{ssec:modsheaves}
Recall that $\bM := \bM_{L}(Q, (2, 3), 5m+1)$ is the moduli space of stable sheaves $F$ on $Q$ with $c_{1}(F) = c_{1}(\cO_{Q}(2, 3))$ and $\chi(F(m)) = 5m + 1$. There are four types of points in $\bM$ (\cite[Theorem 1.1]{Mai16}). Let $C \in |\cO_Q(2,3)|$.
\begin{enumerate}
\item[(0)] $F = \cO_{C}(p+q)$, where the line $\langle p, q\rangle$ is not contained in $Q$;
\item $F = \cO_{C}(p+q)$, where the line $\langle p,q\rangle$ in $Q$ is of type $(1,0)$;
\item $F=\cO_C(0,1)$;
\item $F$ fits into a non-split extension $0 \to \cO_{E} \to F \to \cO_{\ell} \to 0$ where $E$ is a $(2,2)$-curve and $\ell$ is a $(0, 1)$-line.
\end{enumerate}

Let $\bM_{i}$ be the locus of sheaves of the form ($i$). Then $\bM_{i}$ is a subvariety of codimension $i$. $\bM_{1}$ is a $\PP^{9}$-bundle over $\PP^{2} \times \PP^{1}$. $\bM_{2}$ is isomorphic to $|\cO_{Q}(2, 3)|$. Finally, $\bM_{3}$ is a $\PP^{1}$-bundle over $|\cO_{Q}(2, 2)| \times |\cO_{Q}(0, 1)|$. $\bM_{1} \cap \bM_{2}=\bM_{1} \cap \bM_{3}= \emptyset$, but $\bM_{23} := \bM_{2} \cap \bM_{3} \cong |\cO_{Q}(2, 2)| \times |\cO_{Q}(0, 1)|$ (\cite[Theorem 1.1]{Mai16}). Note that $\dim \rH^{0}(F) = 1$ in general, but $\bM_{2}$ parametrizes sheaves that $\dim \rH^{0}(F) = 2$. 

\subsection{Moduli spaces of stable pairs}\label{ssec:moduliofpairs}
A pair $(s, F)$ consists of $F \in \textsf{Coh}(Q)$ and a section $\cO_{Q} \stackrel{s}{\to} F$. Fix $\alpha \in \QQ_{> 0}$. A pair $(s, F)$ is called \emph{$\alpha$-semistable} (resp. \emph{$\alpha$-stable}) if $F$ is pure and for any proper subsheaf $F'\subset  F$, the inequality
\[
	\frac{P(F')(m)+\delta\cdot\alpha}{r(F')} \le (<) 
	\frac{P(F)(m))+\alpha}{r(F)}
\]
holds for $m\gg 0$. Here $\delta=1$ if the section $s$ factors through $F'$ and $\delta=0$ otherwise. Let $\bM^{\alpha} := \bM_{L}^{\alpha}(Q, (2,3), 5m+1)$ be the moduli space of $S$-equivalence classes of $\alpha$-semistable pairs whose support have a class $c_{1}(\cO_{Q}(2, 3))$ (\cite[Theorem 4.12]{LP93b} and \cite[Theorem 2.6]{He98}). The extremal case that $\alpha$ is sufficiently large (resp. small) is denoted by $\alpha=\infty$ (resp. $\alpha = +$). The deformation theory of pairs is studied in \cite[Corollary 1.6 and Corollary 3.6]{He98}.
\begin{proposition}\label{prop:stablepairs}
\begin{enumerate}
\item There exists a natural forgetful map $r:\bM^{+}\longrightarrow \bM$ which maps $(s,F)$ to $F$.
\item (\cite[Section 4.4]{He98}) The moduli space $\bM^{\infty}$ of $\infty$-stable pairs is isomorphic to the relative Hilbert scheme of two points on the complete linear system $|\cO_Q(2,3)|$.
\end{enumerate}
\end{proposition}

The birational map $\bM^{\infty} \dashrightarrow \bM^{+}$ is analyzed in \cite[Theorem 5.7]{Mai16}. It turns out that this is a single flip over $\bM^4$ and is a composition of a smooth blow-up and a smooth blow-down. The blow-up center $\bM_3^{\infty}$ is isomorphic to a $\PP^2$-bundle over $|\cO_{Q}(2,2)|\times |\cO_{Q}(0,1)|$ where a fiber $\PP^2$ parameterizes two points lying on a $(0, 1)$-line. After the flip, the flipped locus on $\bM^{+}$ is $\bM_{3}^{+}$. 

For the forgetful map $r : \bM^{+} \to \bM$, we define $\bM_{i}^{+} := r^{-1}(\bM_{i})$ if $i \ne 3$ and $\bM_{3}^{+}$ is the proper transform of $\bM_{3}$. It contracts $\bM_{2}^+$, which is a $\PP^{1}$-bundle over $\bM_{2}$ and $\bM^{+}\setminus \bM_{2}^{+} \cong \bM \setminus \bM_{2}$. Maican proved that $r$ is a smooth blow-up along the Brill-Noether locus $\bM_{2}$ (\cite[Proposition 5.8]{Mai16}).

\section{Decomposition of the birational map between $\bM$ and $\bP$}\label{sec:masterspace}

In this section we prove Proposition \ref{prop:downstairmodificaiton} and Theorem \ref{thm:mainpropintro} by describing the birational map between $\bM$ and $\bP$. 

\subsection{Construction of a birational map $\bM^+\dashrightarrow \bP$}

\begin{lemma}\label{lem:rigidmap}
There exists a surjective morphism $w: \bM^+\longrightarrow \bG$ which maps $(s,\cO_{C}(p+q)) \in \bM_{0}^{+}$ to $I_{\{p, q\}}(2, 3)$, maps $(s, \cO_{C}(p+q)) \in \bM_{1}^{+}$ to the line $\langle p, q\rangle$ of the type $(1, 0)$, maps $(s, F) \in \bM_{2}^{+}$ to a $(0, 1)$-line determined by a section, and maps $(s, F) \in \bM_{3}^{+}$ to $\ell$ (see Section \ref{ssec:modsheaves} for the notation), a $(0, 1)$-line. 
\end{lemma}

\begin{proof}
By Proposition \ref{prop:stablepairs}, $\bM^{\infty}$ is the relative Hilbert scheme of 2 points on the universal $(2,3)$-curves, which is a $\PP^{9}$-bundle over $\bH(2)$ (\cite[Lemma 2.3]{CC12}). By composing with $t : \bH(2)\lr \bG$ in Proposition \ref{prop:hilbcon}, we have a morphism $\bM^{\infty}\to\bG$. On the other hand, since the flip $\bM^{\infty} \to \bM^{+}$ is the composition of a single blow-up/down, the blown-up space $\widetilde{\bM}^{\infty}$ admits two morphisms to $\bM^{\infty}$ and $\bM^+$, and the flipped locus is $\bM_{3}^{+}$. Note that each point in $\bM_{3}^{+}$ can be regarded as a collection of data $(E, \ell, e)$ where $E$ is a $(2,2)$-curve, $\ell$ is a $(0, 1)$-line, and $e \in \PP\Ext^{1}(\cO_{\ell}, \cO_{E})$. The fiber $\widetilde{\bM}^{\infty} \to \bM^{+}$ over the point in the blow-up center $\bM_{3}^{+}$ is a $\PP^2$ which parameterizes two points on $\ell$. The composition map $\widetilde{\bM}^{\infty}\to \bM^{\infty} \to \bG$ is constant along the $\PP^2$, because $\bG$ does not remember points on the line $\ell \subset Q$. By the rigidity lemma, $\widetilde{\bM}^{\infty}\lr \bG$ factors through $\bM^{+}$ and we obtain a map $w: \bM^{+}\lr \bG$.
\end{proof}

Note that $\bM_{1}^{+} \cong \bM_{1}$ is a $\PP^{9}$-bundle over $\PP^{2} \times \PP^{1}$ and $\bM_{2}^{+}$ is a $\PP^{1}$-bundle over $|\cO_{Q}(2, 3)| \cong \PP^{11}$. They are disjoint divisors on $\bM^{+}$. 

\begin{proposition}\label{prop:mainprop}
There is a birational morphism $q : \bM^{+}\setminus \bM_{1}^{+} \to \bP = \PP(\cU)$ such that $p \circ q : \bM^{+} \setminus \bM_{1}^{+}\to \bP \to \bG$ coincides with $w|_{\bM^{+}\setminus \bM_{1}^{+}}$ in Lemma \ref{lem:rigidmap}. Furthermore, $q$ is the smooth blow-down along $\bM_{2}^{+}$.
\end{proposition}

The proof consists of several steps. Since $\bP = \PP(\cU)$ is a projective bundle over $\bG$, it is sufficient to construct a surjective homomorphism $w^{*}\cU^{*} \to \cL \to 0$ over $\bM^{+} \setminus \bM_{1}^{+}$ for some $\cL \in \mathrm{Pic}(\bM^{+}\setminus \bM_{1}^{+})$, or equivalently, a \emph{bundle} morphism $0 \to \cL^{*} \to w^{*}\cU$. 

Recall that a family $(\cL, \cF)$ of pairs on a scheme $S$ is a collection of data $\cL \in \mathrm{Pic}(S)$, $\cF \in \mathsf{Coh}(S \times Q)$, which is a flat family of pure sheaves, and a surjective morphism $\cExt^2_{\pi}(\cF, \omega_{\pi})\twoheadrightarrow \cL$ where $\pi : S \times Q \to S$ is the projection and $\omega_{\pi}$ is the relatively dualizing sheaf (See \cite[Section 4.3]{LP93b} for the explanation why we take the dual.). Now let $(\cL,\cF)$ be the universal pair (\cite[Theorem 4.8]{He98}) on $\bM^+\times Q$. By applying $\cHom(-,\cO)$ to $\cExt^2_{\pi}(\cF, \omega_{\pi})\twoheadrightarrow \cL$, we obtain $0 \to \cL^{*} \to \cHom(\cExt^{2}_{\pi}(\cF, \omega_{\pi}), \cO)$. It can be shown that $\cHom(\cExt^{2}_{\pi}(\cF, \omega_{\pi}), \cO) \cong \cExt^{1}_{\pi}(\cExt^{1}(\cF, \cO), \cO)$ (see \cite[Section 3.2]{CM15}). So we have a non-zero element $e \in \Hom(\cL^{*}, \cExt_{\pi}^{1}(\cExt^1(\cF,\cO), \cO)) \cong \Ext^{1}(\cExt^1(\cF,\cO), \pi^{*}\cL)$ (\cite[Section 3.2]{CM15}), which provides $\ses{{\pi}^*\cL}{\cE}{\cExt^1(\cF, \cO)}$ on $\bM^{+} \times Q$. By taking $\cHom_{\pi}(-,\omega_{\pi})$, we have $\cExt_{\pi}^2(\cE,\omega_{\pi})\rightarrow \cExt_{\pi}^{2}(\pi^{*}\cL, \omega_{\pi}) \cong \cL^{*}\to 0$ because $\cL$ is a line bundle. This implies the existence of a flat family of pairs $(\cL^*,\cE)$ on $\bM^+\times Q$. We may explicitly describe this construction fiberwisely in the following way. Let $(s, F) \in \bM^{+}$. Let $F^D:=\cExt^1(F, \omega_{Q})$. For a non-zero section $ s\in \rH^0(F)\cong\rH^1(F^D)^{*} \cong\Ext^1(F^D(2,2),(s^{*})\otimes \cO_{Q})$, we have a pair $(s^*,G)$ given by 
\begin{equation}\label{eq-1}
\ses{(s^{*})\otimes \cO_{Q}}{G}{F^{D}(2,2)}.
\end{equation}

\begin{lemma}\label{lem:destab}
The map $(s, F) \mapsto (s^{*},G)$ defines a dominant rational map $\bM^{+} \dashrightarrow \bP = \PP(\cU)$, which is regular on $\bM^{+} \setminus (\bM_{1}^{+} \sqcup \bM_{2}^{+})$.
\end{lemma}

\begin{proof}
Since we have a relative construction of pairs, it suffices to describe the extension $(s^*,G)$ set theoretically. If $(s,F)\in \bM_{0}^{+} \sqcup \bM_{1}^{+}$, then $F\cong \cO_{C}(p+q) \cong I_{Z,C}^D(0,-1)$ for some curve $C$ and $Z = \{p, q\}\in \bH(2)$ such that the line $\ell_{Z}$ containing $Z$ is not in $Q$ (\cite[Section 4.4]{He98}). Then $F^D(2,2)\cong I_{Z,C}(2,3)$. Since $\Ext^{1}(F^{D}(2,2), \cO_{Q}) \cong \rH^1(F^D)^*\cong \rH^0(F) \cong\CC$, from $\ses{\cO_{Q}(-2,-3)\cong I_{C,Q}}{I_{Z,Q}}{I_{Z,C}}$, we obtain $G=I_{Z,Q}(2,3)$. If $(s, F) \in \bM_{0}^{+}$, then we have an element $(s^*, G)\in \bP$ because $G$ has a resolution of the form $\cO_{Q}(0, 1) \to \cO_{Q}(1, 2)^{\oplus 2}$. However, if $(s, F) \in \bM_{1}^{+}$, then we have $\ses{I_{\ell_{Z}, Q}(2, 3)}{G = I_{Z, Q}(2, 3)}{I_{Z, \ell_{Z}}(2, 3)}$ and $I_{\ell_{Z}, Q}(2, 3) = \cO_{Q}(1, 3)$, $I_{Z, \ell_{Z}}(2, 3) = \cO_{\ell_{Z}}(1)$. In particular, $\Hom(\cO_{Q}(1, 3), G) \ne 0$ and $G$ does not admit a resolution $\cO_{Q}(0, 1) \to \cO_{Q}(1, 2)^{\oplus 2}$. So $G \notin \bG$.

Suppose that $(s,F)\in \bM_{3}^{+}\setminus \bM_{2}^+$. Then $F$ fits into a non-split extension $\ses{\cO_E}{F}{\cO_\ell}$. Apply $\cHom(-,\omega_{Q})$, then we have $\ses{\cO_{\ell}(0,1)}{F^{D}(2,2)}{\cO_{E}(2,2)}$. Since $\Ext^1(\cO_E(2,2),\cO_{Q})\cong \Ext^1(F^D(2,2),\cO_{Q}) \cong\CC$, the sheaf $G$ is given by the pull-back:
\begin{equation}\label{eqn:ladder1}
\xymatrix{0\ar[r]&\cO_{Q}\ar[r]\ar@{=}[d]
&\cO_{Q}(2,2)\ar[r]&\cO_E(2,2)\ar[r]&0\\
0\ar[r]&\cO_{Q}\ar[r]&G\ar[r]\ar@{-->}[u]&F^D(2,2)\ar[u]\ar[r]&0
}
\end{equation}
By applying the snake lemma to \eqref{eqn:ladder1}, we conclude that the unique non-split extension $G$ lies on $\ses{\cO_\ell(0,1)}{G}{\cO_{Q}(2,2)}$. Hence $G \in \bG$ (Proposition \ref{prop:hilbcon}) and we have an element $(s^*, G)\in \bP$.

Now suppose that $(s,F)\in \bM_2^+$, so $F=\cO_C(0,1)$. Then $F^D(2,2)=\cO_C(2,2)$. So we have $\ses{(s^{*})\otimes \cO_{Q}}{G}{\cO_C(2,2)}$. By the snake lemma (Consult the proof of \cite[Lemma 3.7]{CM15}.), $G$ fits into $\ses{\cO_{Q}(2,2)}{G}{\cO_{\ell}}$ where $\ell$ is the line of type $(0,1)$ determined by the section $s$. So $\Hom(\cO_{Q}(2, 2), G) \ne 0$ and this implies $G$ does not admit a resolution $\cO_{Q}(0, 1) \to \cO_{Q}(1, 2)^{\oplus 2}$. Thus the correspondence is not well-defined on $\bM_2^+$.
\end{proof}

\subsection{The first elementary modification and the extension of the domain} 

We can extend the morphism in Lemma \ref{lem:destab} by applying an elementary modification of pairs (\cite[Section 2.2]{CC12}) on $\bM_{2}^{+}$. 

\begin{lemma}
There exists an exact sequence of pairs $\ses{(0, K)}{(\cL^*|_{\bM_2^+},\cE|_{\bM_2^+\times Q} )}{(\cL^{''},\cO_{\cZ})}$ where $\cZ$ is the pull-back of the universal family of $(0, 1)$-lines to $\bM_{2}^{+} \times Q$ and $K_{\{m\}\times Q}\cong \cO_Q(2,2)$ for $m = [(s, F)]\in \bM_2^+$.
\end{lemma}

\begin{proof}
The last part of the proof of Lemma \ref{lem:destab} tells us that there is an exact sequence of \emph{sheaves} $\ses{K}{\cE|_{\bM_{2}^{+} \times Q}}{\cO_{\cZ}}$. Now it is sufficient to show that for each fiber $G = \cE|_{\{(s, F)\} \times Q}$, the section $s^{*}$ of $G$ does not come from $\rH^{0}(\cO_{Q}(2, 2))$. If it is, we have an injection $\cO_{Q} \subset \cO_{Q}(2, 2)$ whose cokernel is $\cO_{E}(2, 2)$ for some elliptic curve $E$. By the snake lemma once again, we obtain $\ses{\cO_{E}(2, 2)}{F^{D}(2, 2) = \cO_{C}(2, 2)}{\cO_{\ell}}$. It violates the stability of $F^{D}(2, 2)$. 
\end{proof}


Let $(\cL',\cE')$ be the elementary modification of $(\cL^{*}, \cE)$ along $\bM_{2}^{+}$, that is, 
\[
	\mathrm{Ker}( (\cL^*, \cE)\twoheadrightarrow  
	(\cL^*|_{\bM_2^+},\cE|_{\bM_2^+\times Q} )\twoheadrightarrow 
	(\cL^{''},\cO_{\cZ})).
\]

\begin{lemma}\label{lem:modified}
For a point $m = [(s, F = O_C(0,1))] \in \bM_2^+$, the modified pair $(\cL',\cE')|_{\{m\}\times Q}$ fits into a non-split exact sequence $\ses{(s',\cO_{\ell})}{(s',\cE'|_{\{m\}\times Q})}{(0,\cO_{Q}(2,2))}$ where $\ell$ is a $(0, 1)$-line.
\end{lemma}

\begin{proof}
An elementary modification of pairs interchanges the sub pair with the quotient pair (\cite[Lemma 4.24]{He98}). Thus we obtain the sequence. It remains to show that the sequence is non-split. We will show that the normal bundle $\cN_{\bM_2^+/\bM^{+}}$ at $m$ is canonically isomorphic to $\rH^0(\cO_{\ell})^*$. Then the element $m$ corresponds to the projective equivalent class of nonzero elements in $\rH^{0}(\cO_{\ell})^{*} \cong \Ext^1((0,\cO_{Q}(2,2)),(s',\cO_{\ell}))$, so it is non-split.

The pair $(s,F)$ fits into $\ses{(0,\cO_{Q}(-2,-2))}{(s,\cO_{Q}(0,1))}{(s,F)}$. Thus we have
\[
0\lr \Ext^0((0,\cO_{Q}(-2,-2)), (s,F))\lr \Ext^1((s,F),(s,F))\lr \Ext^1((s,\cO_{Q}(0,1)),(s,F))\lr \cdots.
\]
The first term $\Ext^0((0,\cO_{Q}(-2,-2)), (s,F))\cong \rH^{0}(\cO_{C}(2, 3)) \cong \CC^{11}$ is the deformation space of curves $C$ on $Q$. The second term $\Ext^{1}((s, F), (s,F))$ is $\cT_{m}\bM^{+}$ (\cite[Theorem 3.12]{He98}). For the third term, by \cite[Theorem 3.12]{He98}, we have
\[
0\lr \Hom(s, \rH^0(F)/\langle s\rangle)\lr \Ext^1((s,\cO_{Q}(0,1)),(s,F))\lr \Ext^1(\cO_{Q}(0,1), F)\stackrel{\phi}{\lr} \Hom(s, \rH^1(F)).
\]
The first term $\Hom(s, \rH^0(F)/\langle s\rangle)=\CC$ is the deformation space of the line $\ell$ in $Q$ determined by the section $s$. By Serre duality, $\phi : \rH^{0}(\cO_{Q}(0, 1))^{*} \to \rH^{0}(\cO_{Q})^{*}$ and the kernel is $\rH^{0}(\cO_{\ell}(0, 1))^{*} \cong \rH^{0}(\cO_{\ell})^{*}$. This proves our assertion.
\end{proof}

Recall that the modified pair $(\cL',\cE')$ provides a natural surjection $\cExt_{\pi}^2(\cE',\omega_{\pi})\twoheadrightarrow \cL'$ on $\bM^+ \times Q$. It is straightforward to check that $\cExt_{\pi}^2(\cE',\omega_{\pi})$ has rank $10$ at each fiber, thus it is locally free.

\begin{proof}[Proof of Proposition \ref{prop:mainprop}]
We claim that there exists a surjection $w^*\cU^*\lr \cL' \lr 0$ up to a twisting by a line bundle on $\bM^+\setminus \bM_1^+$. Then there is a morphism $\bM^+\setminus \bM_1^+\lr \bP$.

Consider the following commutative diagram
\[
\xymatrix{(\bM^+\setminus \bM_1^+)\times Q\ar[r]_{w':=w\times \mathrm{id}}\ar[d]_{\pi}& \bG\times Q\ar[d]^{\pi}\\
\bM^+\setminus \bM_1^+ \ar[r]^{w}&\bG.}
\]
Note that $\cU = \pi_*(\cW)$ where $\cW=\mathrm{coker}(\phi)$ is the universal quotient on $\bG\times Q$ (Section \ref{sub:defq}). One can check that $\cW$ is flat over $\bG$. By its construction of $w$, $\cE'|_{\{m\}\times Q}\cong {w'}^{*}\cW|_{\{m\}\times Q}$ restricted to each point $m\in \bM^+\setminus \bM_{1}^{+}$. The universal property of $\bG$ (as a quiver representation space \cite[Proposition 5.6]{Kin94}) tells us that ${w'}^{*}\cW\cong \cE'$ up to a twisting by a line bundle on $\bM^+\setminus \bM_1^+$. The base change property implies that there exists a natural isomorphism (up to a twisting by a line bundle) $w^*\cU=w^*(\pi_*\cW)\cong \pi_*({w'}^*\cW)=\pi_* \cE' \cong \cExt_\pi^2(\cE',\omega_{\pi})^*$ by \cite[Corollary 8.19]{LP93b}. Hence we have $w^*\cU^*\cong (w^*\cU)^*\cong (\pi_* (\cE'))^*\cong \cExt_{\pi}^2(\cE',\omega_{\pi}) \twoheadrightarrow \cL'$. Therefore we obtain a morphism $q : \bM^{+}\setminus \bM_1^+ \to \bP$. 

By the proof of Lemma \ref{lem:modified}, the modified pair does not depend on the choice of a $(2, 3)$-curve, so $q: \bM^{+}\setminus \bM_1^+\to \bP\setminus p^{-1}(t(Y_{10}))$ is indeed a contraction of $\bM_{2}^{+}$ and the image of $\bM_{2}^{+}$ is $Y_{01}$. Recall that the exceptional divisor $\bM_{2}^{+}$ is $|\cO_{Q}(2, 3)| \times |\cO_{Q}(0, 1)| \cong \PP^{11} \times \PP^{1}$. Note that the sheaf $F$ in the pair $(s, F) \in \bM_{2}^{+}$ is parametrized by $\PP^{11} = |\cO_{Q}(2, 3)| = \PP\Ext^{1}(\cO_{Q}(-2,-2)[1], \cO_{Q}(0, 1))$. It follows also from the fact that each $F$ fits into a triangle $\ses{\cO_{Q}(0, 1)}{F}{\cO_{Q}(-2,-2)[1]}$. By analyzing $T_{F}\bM = \Ext^{1}(F, F)$ (which is similar to \cite[Lemma 3.4]{CC12}), one can see that $\cN_{\bM_{2}/\bM}|_{\PP^{11}} \cong \Ext^{1}(\cO_{Q}(0, 1), \cO_Q(-2, -2)[1])\otimes \cO_{\PP^{11}}(-1) \cong \rH^{0}(\cO_{Q}(0, 1))^* \otimes \cO_{\PP^{11}}(-1)$. Thus $\cN_{\bM_{2}^{+}/\bM^{+}} \cong \cO_{\PP^{11}\times \PP^{1}}(-1, -1)$ and $q$ is a smooth blow-down by Fujiki-Nakano criterion. 
\end{proof}

Thus we have two different contractions of $\bM^{+}$, one is $\bM$ obtained by contracting all $\PP^{1}$-fibers on $\bM_{2}^{+}$, and the other is:

\begin{definition}\label{def:mminus}
Let $\bM^{-}$ be the contraction of $\bM^+$ which is obtained by contracting all $\PP^{11}$-fibers on $\bM_{2}^{+}$. We define $\bM_{i}^{-}$ as the image of $\bM_{i}^{+}$ for the contraction $\bM^{+} \to \bM^{-}$.
\end{definition}

\subsection{The second elementary modification and $\bM^{-}$}\label{ssec:modification}

Recall that $u: \bG_1\lr \bG$ is the blow-up of $\bG$ along the $\PP^1$ parameterizing $(1,0)$-lines in $Q$, and $Y_{10}$ is the exceptional divisor. Let $\cW$ be the cokernel of the universal morphism $\phi$ on $\bG\times Q$ in Section \ref{sub:defq}. Let $\cV:=(u\times \mathrm{id})^*\cW$ be the pull-back of $\cW$ along the map $u\times \mathrm{id}:\bG_1\times Q\lr \bG\times Q$. Then for $([\ell],t)\in Y_{10}$, $\cV|_{([\ell], t) \times Q}$ fits into a non-split exact sequence $\ses{\cO_{\ell}(1)}{\cV|_{([\ell],t)\times Q}}{\cO_Q(1,3)}$. By relativizing it over $Y_{10} \times Q$, we obtain $\ses{\cS}{\cV|_{Y_{10}\times Q}}{\cQ}$. Let $\cV^{-}$ be the elementary modification $\mathrm{elm}_{Y_{10}\times Q}(\cV,\cQ):=\mathrm{ker}(\cV\twoheadrightarrow \cV|_{Y_{10}\times Q}\twoheadrightarrow \cQ)$ along $Y_{10} \times Q$. Note that over $([\ell], t) \in \bG_{1}$, $\cV^{-}|_{([\ell], t)\times Q}$ fits into a non-split exact sequence $\ses{\cO_{Q}(1, 3)}{\cV^{-}|_{([\ell], t)\times Q}}{\cO_{\ell}(1)}$ because the elementary modification interchanges the sub/quotient sheaves. Let $\pi_1: \bG_1\times Q \lr \bG_1$ be the projection into the first factor. Then $\cU^{-} := \pi_{1*}\cV^{-}$ is a rank $10$ bundle over $\bG_1$. Let $\bP^{-} := \PP(\cU^{-})$. 

The following proposition completes the proof of Theorem \ref{thm:mainpropintro}.

\begin{proposition}
The projective bundle $\bP^{-}$ is isomorphic to $\bM^{-}$ in Definition \ref{def:mminus}.
\end{proposition}

\begin{proof}
Since the elementary modification has been done locally around $Y_{10} \times Q$, $\PP(u^{*}\cU)$ and $\bP^-$ are isomorphic over $\bG_1\setminus Y_{10}$. On the other hand, set theoretically, it is straightforward to see that the image of $q$ is $\bP \setminus p^{-1}(t(Y_{10}))$, where $p : \bP \to \bG$ is the structure morphism. So we have a birational morphism $\bM^{+}\setminus \bM_{1}^{+} \to \bP \setminus p^{-1}(t(Y_{10})) \cong \PP (u^{*}\cU) \setminus p^{-1}(Y_{10}) \cong \bP^{-} \setminus p^{-1}(Y_{10})$ (here we used the same notation $p$ for the projections $\PP(u^{*}\cU) \to \bG_{1}$ and $\bP^{-} \to \bG_{1}$). By Proposition \ref{prop:mainprop}, this map is a blow-down along $\bM_{2}^{+}$, thus we have an isomorphism $\tau : \bP^{-}\setminus p^{-1}(Y_{10}) \to \bM^{-}\setminus \bM_{1}^{-}$. So we have a birational map $\tau : \bP^{-} \dashrightarrow \bM^{-}$, where its undefined locus is $p^{-1}(Y_{10})$. 

On the other hand, since the flipped locus for $\bM^{\infty} \dashrightarrow \bM^{+}$ is $\bM_{3}^{+}$, we have an isomorphism $\bM^{-}\setminus (\bM_{2}^{-}\cup \bM_{3}^{-}) \cong \bM^{+}\setminus (\bM_{2}^{+}\cup \bM_{3}^{+}) \cong \bM^{\infty} \setminus (\bM_{2}^{\infty} \cup \bM_{3}^{\infty})$ (Here $\bM_{i}^{\infty}$ is defined in an obvious way.). Also $\tau^{-1}(\bM_{2}^{-}\cup \bM_{3}^{-}) = p^{-1}(Y_{01})$. Hence if we restrict the domain of $\tau$, then we have $\sigma : \bP^{-}\setminus p^{-1}(Y_{01}) \dashrightarrow \bM^{-}\setminus (\bM_{2}^{-}\cup \bM_{3}^{-}) \cong \bM^{\infty}\setminus (\bM_{2}^{\infty}\cup \bM_{3}^{\infty})$ whose undefined locus is $p^{-1}(Y_{10})$. Therefore $\sigma$ can be regarded as a map into a relative Hilbert scheme. Note that $\bM_{2}^{\infty} \cup \bM_{3}^{\infty}$ is the locus of $(2, 3)$-curves passing through two points lying on a $(0, 1)$-line. 

We claim that $\sigma$ is extended to a morphism $\tilde{\sigma} : \bP^{-}\setminus p^{-1}(Y_{01}) \to \bM^{-}$ such that $\tilde{\sigma}(p^{-1}(Y_{10})) = \bM_{1}^{-} \cong \bM_{1}^{\infty}$. To show this, it is enough to check that $\cV^{-}$ over $Y_{10}$ provides a flat family of the twisted ideal sheaf of Hilbert scheme of two points lying on $(1,0)$-type lines. Note that $\cV^{-}$ fits into a non-split extension $\ses{\cO_Q(1,3)}{\cV^{-}|_{([\ell], t)\times Q}}{\cO_{\ell}(1)}$. By a diagram chasing similar to the second paragraph of the proof of Lemma \ref{lem:destab}, one can check that $\cV^{-}|_{([\ell], t)\times Q} \cong I_{Z, Q}(2, 3)$ where $Z\subset \ell$ and $\ell$ is a $(1, 0)$-line.

Now two maps $\tau$ and $\tilde{\sigma}$ coincide over the intersection $\bP^{-} \setminus p^{-1}(Y_{10} \cup Y_{01})$ of domains, so we have a birational morphism $\bP^{-} \to \bM^{-}$. Since $\rho(\bP^{-}) = 3 = \rho(\bM^{-})$ and both of them are smooth, this map is an isomorphism. 
\end{proof}

The modification on $\bG_1\times Q$ descends to $\bG_1$. Then Proposition \ref{prop:downstairmodificaiton} follows from a general result of Maruyama (\cite{Mar73}).

\begin{proof}[Proof of Proposition \ref{prop:downstairmodificaiton}]
Let $\pi_{1} : \bG_{1} \times Q \to \bG_{1}$ be the projection. We claim that $\cU^{-} = \mathrm{elm}_{Y_{10}}(u^{*}\cU, \pi_{1*}\cQ) \cong \pi_{1 *}\mathrm{elm}_{Y_{10}\times Q}(\cV, \cQ)$. Indeed, from $\ses{\cV^{-}}{\cV}{\cQ}$, we have $0\to\pi_{1 *}\cV^{-}\to\pi_{1 *}\cV = u^{*}\cU \to \pi_{1 *}\cQ \to R^{1}\pi_{1 *}\cV^{-}\to R^{1}\pi_{1 *}\cV$. It is sufficient to show that $R^{1}\pi_{1 *}\cV^{-} = 0$. By using the resolution of $\cV$ given by the universal morphism $\phi$, we have $R^{1}\pi_{1 *}\cV = 0$. Over $\bG_{1} \setminus Y_{10}$, the last two terms are isomorphic. Over $Y_{10}$, from $\rH^{1}(\cO_{Q}(1, 3)) = \rH^{1}(\cO_{\ell}(1)) = 0$ and the description of $\cV^{-}|_{([\ell], t)}$, we obtain $R^{1}\pi_{1 *}\cV^{-} = 0$. 

Note that $u^{*}\cU|_{Y_{10}}$ fits into a \emph{vector bundle} sequence $\ses{\pi_{1 *}\cS}{u^{*}\cU|_{Y_{10}}}{\pi_{1 *}\cQ}$ and $\mathrm{rank}\; \pi_{1 *}\cS = 2$ and $\mathrm{rank}\;\pi_{1 *}\cQ = 8$. The result follows from \cite[Theorem 1.3]{Mar73}.
\end{proof}

As a direct application of Theorem \ref{thm:mainpropintro}, we compute the Poincar\'e polynomial of $\bM$ which matches with the result in \cite[Theorem 1.2]{Mai16}.
\begin{corollary}\label{coro}
\begin{enumerate}
\item The moduli space $\bM$ is rational;
\item The Poincar\'e polynomial of $\bM$ is 
\[
P(\bM) =\; q^{13}+3q^{12}+8q^{11}+10q^{10}+11 q^{9}+11 q^{8}+11 q^{7}+11 q^{6}+11q^5+11q^4+10q^3+8q^2+3q+1.
\]
\end{enumerate}
\end{corollary}
\begin{proof}
Now $\bM$ is birational to a $\PP^9$-bundle over $\bG$, so we obtain Item (1). Item (2) is a straightforward calculation using
\[
P(\bM) =\; P(\PP^{11})-P(\PP^1)+P(\bM^-) =P(\PP^{11})-P(\PP^1)+P(\PP^{9})(P(\bG)+(P(\PP^2)-1)P(\PP^1)).
\]
\end{proof}




\begin{thebibliography}{ABCH13}


\bibitem[BC13]{BC13}
Aaron Bertram and Izzet Coskun.
\newblock The birational geometry of the {H}ilbert scheme of points on
  surfaces.
\newblock In {\em Birational geometry, rational curves, and arithmetic}, pages
  15--55. Springer, New York, 2013.

\bibitem[CC16]{CC12}
Jinwon Choi and Kiryong Chung.
\newblock Moduli spaces of $α$-stable pairs and wall-crossing on
  $\mathbb{P}^2$.
\newblock {\em J. Math. Soc. Japan}, 68(2):685--789, 2016.

\bibitem[CM15]{CM15}
Kiryong Chung and Han-Bom Moon.
\newblock Chow ring of the moduli space of stable sheaves supported on quartic
  curves.
\newblock arXiv:1506.00298, To appear in Quarterly Journal of Mathematics,
  2015.

\bibitem[He98]{He98}
Min He.
\newblock Espaces de modules de syst{\`e}mes coh{\'e}rents.
\newblock {\em Internat. J. Math.}, 9(5):545--598, 1998.

\bibitem[HL10]{HuLe10}
Daniel Huybrechts and Manfred Lehn.
\newblock {\em The geometry of moduli spaces of sheaves}.
\newblock Cambridge Mathematical Library. Cambridge University Press,
  Cambridge, second edition, 2010.

\bibitem[Kin94]{Kin94}
A.~D. King.
\newblock Moduli of representations of finite-dimensional algebras.
\newblock {\em Quart. J. Math. Oxford Ser. (2)}, 45(180):515--530, 1994.

\bibitem[LP93]{LP93b}
Joseph Le~Potier.
\newblock Syst{\`e}mes coh{\'e}rents et structures de niveau.
\newblock {\em Ast{\'e}risque}, (214):143, 1993.

\bibitem[Mai16]{Mai16}
Mario Maican.
\newblock Moduli of sheaves supported on curves of genus two in a quadric
  surface.
\newblock {\em arXiv:1612.03566}, 2016.

\bibitem[Mar73]{Mar73}
M.~Maruyama.
\newblock On a family of algebraic vector bundles.
\newblock {\em Number Theory, Algebraic Geometry, and Commutative Algebra},
  pages 95--149, 1973.

\end{thebibliography}

\bibliographystyle{alpha}
\newcommand{\etalchar}[1]{$^{#1}$}

\end{document}